\documentclass[11pt]{amsart}
\usepackage{fullpage}
 
\usepackage{hyperref}
\usepackage{xypic}
\usepackage{enumerate}
\usepackage{amsthm}
\usepackage{amsmath}
\usepackage{amssymb}
\usepackage{amsfonts}
\newtheorem{theo}{Theorem}[section]
\newtheorem{lemma}[theo]{Lemma}
\newtheorem{prop}[theo]{Proposition}
\newtheorem{cor}[theo]{Corollary}

\newtheorem*{theononumber}{Theorem}

\newtheorem{lem}[theo]{Lemma}

\theoremstyle{definition}

\newtheorem{rem}[theo]{Remark}

\newtheorem{defi}[theo]{Definition}

\renewcommand{\labelenumi}{{\rm (\roman{enumi})}}
 
\newcommand{\f}{\phi}

\newcommand{\spec}{\operatorname{Spec}}

\newcommand{\gal}{\operatorname{Gal}}
\newcommand{\Gl}{\operatorname{GL}}

\newcommand{\Aut}{\operatorname{Aut}}

\newcommand{\trdeg}{\operatorname{trdeg}}

\newcommand{\id}{\operatorname{id}}

\newcommand{\C}{\mathbb{C}}

\newcommand{\Kb}{{\bar{K}}}

\newcommand{\Ind}{\operatorname{Ind}}
\newcommand{\del}{\partial}
\newcommand{\Frac}{\ensuremath{\mathrm{Frac}}}

\newcommand{\rank}{\operatorname{rank}}

\newcommand{\GL}{\operatorname{GL}}
\newcommand{\Spec}{\operatorname{Spec}}

\title{The differential Galois group of the rational function field}
\author{Annette Bachmayr, David Harbater, Julia Hartmann and Michael Wibmer}

\date{January 12, 2021}
\thanks{The first author was funded by the Deutsche Forschungsgemeinschaft (DFG) - grants MA6868/1-1, {MA6868/1-2}. The second and third authors were supported on NSF grants DMS-1463733 and DMS-1805439.  The fourth author was supported by the NSF grants DMS-1760212, DMS-1760413, DMS-1760448 and the Lise Meitner grant M 2582-N32 of the Austrian Science Fund FWF.  We also acknowledge support from NSF grant DMS-1952694}


\begin{document}
\maketitle

\begin{abstract}
We determine the absolute differential Galois group of the field $\C(x)$ of rational functions: It is the free proalgebraic group on a set of cardinality $|\C|$. This solves a longstanding open problem posed by B.H. Matzat. For the proof we develop a new characterization of free proalgebraic groups in terms of split embedding problems, and we use patching techniques in order to solve a very general class of differential embedding problems.  Our result about $\C(x)$ also applies to rational function fields over more general fields of coefficients.
\end{abstract}

\section{Introduction}

The differential Galois group of a linear differential equation is a linear algebraic group that measures the algebraic relations among the solutions: the larger the group, the fewer relations there are. This group is central for analyzing algebraic properties of the solutions.

One of the guiding problems in the Galois theory of differential equations is the so-called inverse problem. It asks which linear algebraic groups occur as differential Galois groups over a given differential field $F$. In the classical case $F=\C(x)$, the inverse problem was solved by C. Tretkoff and M. Tretkoff (\cite{TretkoffTretkoff:SolutionOfTheInverseProblemOfDifferentialGaloisTheoryInTheClassicalCase}) based on J. Plemelj's solution of Hilbert's 21st problem: Every linear algebraic group is a differential Galois group over $\C(x)$. 

There are various directions in which the inverse problem can be strengthened. For example,
one may ask
how many essentially different ways a given linear algebraic group $G$ occurs as a differential Galois group over $\C(x)$. More precisely, one may ask to determine the cardinality $\kappa_G$ of the set of isomorphism classes of Picard-Vessiot extensions of $\C(x)$ with differential Galois group isomorphic to $G$. Here a Picard-Vessiot extension is the differential analog of a finite Galois extension of fields in usual Galois theory. The solution of the inverse problem simply states that $\kappa_G\geq 1$ for any linear algebraic group $G$. Since there are only $|\C(x)|=|\C|$ linear differential equations over $\C(x)$, one trivially has $\kappa_G\leq |\C|$. J. Kovacic showed (\cite{Kovacic:TheInverseProblemInTheGaloisTheoryOfDifferentialFields}) that $\kappa_G=|\C|$ when $G$ is a connected solvable linear algebraic group. Our main result (see below) implies that indeed $\kappa_G=|\C|$ for any non-trivial linear algebraic group $G$.

The absolute differential Galois group of a differential field $F$ is a proalgebraic group that governs the linear part of the differential field arithmetic of $F$. 
This group is the projective limit of all differential Galois groups of all linear differential equations with coefficients in $F$. 
A linear algebraic group is a differential Galois group over $F$ if and only if it is a quotient of the absolute differential Galois group of $F$. Thus the problem of determining the absolute differential Galois group can also be seen as a strengthening of the inverse problem.

For the differential fields 
of formal and convergent Laurent series,
the absolute differential Galois group is known (\cite[Chapter 10]{SingerPut:differential}). On the other hand, in the classical case of the rational function field $F=\C(x)$, no explicit description of the absolute differential Galois group has been known so far. During the 1999 MSRI program \emph{Galois Groups and Fundamental Groups}, B.H. Matzat suggested that there should be a differential generalization of A. Douady's theorem (\cite{Douady}; see also \cite[Theorem 3.4.8]{Szamuely:GaloisGroupsAndFundamentalGroups}) that the absolute Galois group of $\C(x)$ is the free profinite group on a set of cardinality $|\C|$. In this paper we show that this is indeed the case (Theorem \ref{result}):

\begin{theononumber}[Matzat's conjecture for $\C(x)$]
	The absolute differential Galois group of $\C(x)$ is the free proalgebraic group on a set of cardinality $|\C|$.	
\end{theononumber}
In fact, our theorem provides another proof of Douady's result as well; see Remark~\ref{rem: Matzat implies Douady}.  

Previously known freeness results for absolute differential Galois groups include the case of prounipotent extensions. Free prounipotent proalgebraic groups were introduced and characterized in terms of embedding problems in \cite{LubotzkyMagid:CohomologyOfUnipotentAndPropunipotentGroups}. Based on this characterization, A.~Magid has recently shown (\cite{Magid:TheDifferentialGaloisGroupOfTheMaximalProunipotentExtensionIsFree}) that the maximal prounipotent quotient of \emph{any} absolute differential Galois group is free.

In \cite{BachmayrHarbaterHartmannWibmer:FreeDifferentialGaloisGroups}, the authors were able to prove the analog of Matzat's conjecture for a countable algebraically closed field of infinite transcendence degree and characteristic zero in place of $\C$.

In this article we in fact prove Matzat's conjecture for any \emph{uncountable} algebraically closed field of characteristic zero in place of $\C$; and so, in combination with loc.~cit., obtain the case of algebraically closed fields of infinite transcendence degree over $\mathbb{Q}$ (Theorem 5.6). We note that the proofs in the countable and the uncountable case are quite different. On the one hand, the characterization of the freeness in terms of embedding problems is much simpler in the countable case; on the other hand, the specialization argument (Proposition \ref{prop: descent for embedding problems}) that we use in the uncountable case only works for uncountable fields. 

Our proof of Matzat's conjecture is based on the study of certain differential embedding problems that we solve using patching techniques. A differential embedding problem of finite type over $\C(x)$ is a Picard-Vessiot extension $L/\C(x)$ with differential Galois group $H$ together with a surjective morphism $G\to H$ of linear algebraic groups. A solution is a Picard-Vessiot extension $M/\C(x)$ with differential Galois group $G$ together with an embedding of $L$ into $M$ that is compatible with the morphism $G\to H$. Note that the special case $H=1$ corresponds to the inverse problem. Solutions to differential embedding problems tell us not only which linear algebraic groups occur as differential Galois groups but also how these groups fit together in the projective system that defines the absolute differential Galois group.

 It was shown in \cite{BachmayrHarbaterHartmannWibmer:DifferentialEmbeddingProblems} (and in \cite{BachmayrHartmannHarbaterPopLarge} for more general fields than $\C$) that every differential embedding problem of finite type over $\C(x)$ has a solution. 
However, to establish Matzat's conjecture a significantly stronger result is needed. In fact, it was shown in \cite{BachmayrHarbaterHartmannWibmer:FreeDifferentialGaloisGroups} that Matzat's conjecture is equivalent to the statement that every differential embedding problem of finite type over $\C(x)$ has $|\C|$ independent solutions; i.e., there exist $|\C|$ solutions that are linearly disjoint over the given extension~$L$.

Another path to establish Matzat's conjecture via solving differential embedding problems is to relax the finite type assumption, and this is the one we will follow in this article. Instead of a surjective morphism $G\to H$ of linear algebraic groups, one allows a surjective morphism of proalgebraic groups with algebraic kernel.  And instead of taking a single Picard-Vessiot extension $L/\C(x)$, one allows a direct limit of Picard-Vessiot extensions, where the limit it taken over a directed set of cardinality less than $|\C|$. We call differential embedding problems of this kind admissible. In fact, Matzat's conjecture is also equivalent to the statement that all admissible differential embedding problems over $\C(x)$ have a solution (loc.~cit.). Two key ingredients of our proof of the latter condition are:

\begin{itemize}
	\item A new characterization of free proalgebraic groups in terms of embedding problems that relies on and improves the characterization from \cite{Wibmer:FreeProalgebraicGroups}. The crucial point here is that to prove Matzat's conjecture we can reduce to solving split admissible differential embedding problems; i.e., we can assume that the morphism $G\to H$ has a section.
	\item To solve these split admissible differential embedding problems, we generalize the patching techniques from \cite{BachmayrHarbaterHartmannWibmer:DifferentialEmbeddingProblems} and \cite{BachmayrHartmannHarbaterPopLarge}, where only differential embedding problems of finite type were considered.
\end{itemize}

The idea to prove the freeness of a Galois group by solving embedding problems goes back to K.~Iwasawa (\cite{Iwasawa:OnSolvableExtensionsOfAlgebraicNumberFields}), who used it to prove the solvable version of a conjecture of I.~Shafarevich in inverse Galois theory.  Shafarevich's conjecture (see \cite{Harbater:ShafarevicConjecture}) states that the absolute Galois group of the maximal abelian extension $\mathbb{Q}^{\text{ab}} = \mathbb{Q}^{\text{cycl}}$ of $\mathbb{Q}$ is a free profinite group on a countably infinite set.  The function field analog of that conjecture (proven in \cite{Har95} and \cite{Pop95}) asserts the freeness of the absolute Galois group of ${\bar{\mathbb{F}}}_p(x) = \mathbb{F}_p(x)^{\text{cycl}}$, and more generally of $k(x)$ for $k$ algebraically closed of characteristic $p$.  That result, combined with the parallel between differential Galois theory in characteristic zero and usual Galois theory in characteristic $p$ (see \cite[Section~11.6]{SingerPut:differential}), provided motivation for Matzat's conjecture.

Our main theorem (Matzat's conjecture for ${\mathbb C}(x)$) can also be reformulated in the Tannakian framework.  Namely, from that perspective, it states that there is an equivalence of Tannakian categories, between the category of finite dimensional differential modules over $\C(x)$ and that of cofinite representations of the free group on $|\C|$ elements.

\medskip

{\bf Structure of the manuscript.} In Section~\ref{ebp}, after recalling the definition of free proalgebraic groups, we provide a new characterization of these groups in terms of algebraic embedding problems and split admissible embedding problems. In Section~\ref{sect: dgt} we recall basic definitions and results from differential Galois theory and show that to prove Matzat's conjecture, it suffices to solve split admissible differential embedding problems. In Section~\ref{sec: Solving differential embedding problems} we establish a patching setup that we then use in Section~\ref{sec: Matzat} to prove Matzat's conjecture. 

{\bf Acknowledgments.} The authors thank A. Pillay for helpful discussions.

\section{Free proalgebraic groups and embedding problems}\label{ebp}

In this section we first recall the definition of free proalgebraic groups and their characterization in terms of embedding problems from \cite{Wibmer:FreeProalgebraicGroups}. We then show that, in this context, it suffices to consider algebraic embedding problems and certain split embedding problems. 

Throughout this section we work over an arbitrary base field $K$ with algebraic closure $\Kb$.

\subsection{Proalgebraic groups} \label{proalg gps}

We begin by introducing some notation for proalgebraic groups. To simplify terminology we will use the term \emph{algebraic group} (over $K$) to mean \emph{affine group scheme of finite type} (over $K$) and the term \emph{proalgebraic group} (over $K$) to mean \emph{affine group scheme} (over $K$). 
Equivalently, a proalgebraic group is a projective limit of algebraic groups. A closed normal subgroup $N$ of a proalgebraic group $G$ over $K$ is called \emph{coalgebraic} if $G/N$ is algebraic. 

The coordinate ring (i.e., the ring of global functions) on a proalgebraic group $G$ is denoted by $K[G]$, so $G=\spec(K[G])$. We also write $K[X]$ for the global functions on an affine $K$-scheme $X$.   It is often convenient to identify a proalgebraic group $G$ with its functor of points $T\mapsto G(T)$, for $T$ a $K$-algebra. The scheme theoretic image of a morphism $\f\colon G\to H$ of proalgebraic groups is denoted by $\f(G)$. It is the smallest closed subgroup of $H$ such that $\f$ factors through the inclusion $\f(G)\to H$.

A morphism $\f\colon G\to H$ of proalgebraic groups is an \emph{epimorphism} if $\f(G)=H$. This is equivalent to the dual map $\f^*\colon K[H]\to K[G]$ being injective. It is also equivalent to the following statement: For every $K$-algebra $R$ and $h\in H(R)$, there exists a faithfully flat $R$-algebra $S$ and $g\in G(S)$ such that under $\f$ the element $g$ maps to the image of $h$ in $H(S)$.

If $G$ is a proalgebraic group over $K$ and if $F/K$ is a field extension, then we may consider the base change $G_F := F \times_K G$ and its coordinate ring $F[G_F] = F \otimes_K K[G]$.  We often suppress the base change subscript (e.g., writing $F[G]$ for $F[G_F]$), except when necessary to avoid confusion.

\medskip

A \emph{free proalgebraic group over $K$ on a set $X$} is a proalgebraic group $\Gamma=\Gamma_K(X)$ equipped with a map $\iota\colon X\to \Gamma(\Kb)$ such that 
\begin{itemize}
\item for every coalgebraic subgroup $N$ of $\Gamma$ all but finitely many elements of $X$ map into $N(\Kb)$ under $\iota$ and
\item whenever $G$ is a proalgebraic group over $K$ and $\varphi\colon X\to G(\Kb)$ is a map as above, then there exists a unique morphism $\psi\colon \Gamma\to G$ such that
$$
\xymatrix{
	X \ar^\iota[rr] \ar_\varphi[rd] & & \Gamma(\Kb) \ar^{\psi_\Kb}[ld] \\
	& G(\Kb) &	
}
$$
commutes. 
\end{itemize}
As usual, this is unique up to isomorphism. For our proof of Matzat's conjecture we will use a different characterization of free proalgebraic groups. To this end we recall some definitions from \cite{Wibmer:FreeProalgebraicGroups}.

\begin{defi}
	Let $G$ be a proalgebraic group that is not algebraic. The \emph{rank} of $G$, denoted by $\rank(G)$, is the dimension of $K[G]$ as a $K$-vector space. For a non-trivial algebraic group $G$ we set $\rank(G)=1$.  The rank of the trivial group is defined to be zero.
\end{defi}

See \cite[Prop. 3.1]{Wibmer:FreeProalgebraicGroups} for different characterizations of the rank of a proalgebraic group. If $X$ is a set with $|X|\geq |K|$ and $\operatorname{char}(K)=0$, then $\Gamma_K(X)$ is of rank $|X|$ (\cite[Cor. 3.12]{Wibmer:FreeProalgebraicGroups}). 

\begin{defi}
	Let $\Gamma$ be a proalgebraic group. An \emph{embedding problem} for $\Gamma$ consists of epimorphisms $\alpha\colon G\twoheadrightarrow H$ and $\beta\colon \Gamma\twoheadrightarrow H$ of proalgebraic groups. A \emph{(proper) solution} is an epimorphism $\f\colon\Gamma\twoheadrightarrow G$ such that $\beta=\alpha\f$, i.e.,
	$$
	\xymatrix{
		\Gamma \ar@{->>}^\beta[rd] \ar@{..>>}_\f[d]& \\
		G \ar@{->>}^-\alpha[r] & H	
	}
	$$
	commutes. A \emph{weak solution} is a morphism (not necessarily an epimorphism) $\f\colon \Gamma\to G$ such that $\beta=\alpha\f$. If $G$ (and therefore also $H$) is algebraic, the embedding problem is called \emph{algebraic}. 
	
	The embedding problem is \emph{split} if $\alpha\colon G\twoheadrightarrow H$ splits, i.e., if there exists a morphism $\alpha'\colon H\to G$ such that $\alpha\alpha'=\id_H$.

	The \emph{kernel} of the embedding problem is the kernel of $G\twoheadrightarrow H$. If the kernel is algebraic (i.e., of finite type over $K$) and $\rank(H)<\rank(\Gamma)$, the embedding problem is \emph{admissible}.
\end{defi}

\begin{defi}
	A proalgebraic group $\Gamma$ is \emph{saturated} if every admissible embedding problem for $\Gamma$ has a solution.
\end{defi}

Theorem 3.24 of \cite{Wibmer:FreeProalgebraicGroups} provides several equivalent characterizations of saturated proalgebraic groups in terms of embedding problems. Moreover, over a field of characteristic zero, a proalgebraic group $\Gamma$ with $\rank(\Gamma)\geq |K|$ is saturated if and only if it is free on a set of cardinality $\rank(\Gamma)$ (\cite[Theorem 3.42]{Wibmer:FreeProalgebraicGroups}).

\subsection{Reduction to the case of split admissible embedding problems} \label{red split adm}

The following lemma shows that, as far as weak solutions are concerned, one can restrict to the case of algebraic embedding problems.

\begin{lemma} \label{lemma: algebraic suffices for projective}
	Let $\Gamma$ be a proalgebraic group. If every algebraic embedding problem for $\Gamma$ has a weak solution, then every embedding problem for $\Gamma$ has a weak solution.
\end{lemma}

\begin{proof}
First consider a pair consisting of an epimorphism $\alpha\colon G\twoheadrightarrow H$ of algebraic groups and any morphism
$\beta\colon\Gamma\to H$ (not necessarily an epimorphism).
We claim that we can complete the commutative diagram
\begin{equation} \label{eqn: triangle for projective}
	\xymatrix{
		\Gamma \ar^\beta[rd] \ar@{..>}_\f[d]& \\
		G \ar@{->>}^-\alpha[r] & H	
	}
	\end{equation}
with a morphism $\f\colon\Gamma\to G$.  To see this, 
replace $H$ by $\beta(\Gamma)$ and $G$ by $\alpha^{-1}(\beta(\Gamma))$; 
this yields an algebraic embedding problem for $\Gamma$, which has a weak solution, by hypothesis.  We thus obtain the asserted map $\phi$, proving the claim.

By \cite[Prop.~4]{BassLubotzkyMagidMozes:TheProalgebraicCompletionOfRigidGroups}, the above claim implies that every diagram  (\ref{eqn: triangle for projective}) 
	with $\alpha\colon G\twoheadrightarrow H$ an epimorphism of proalgebraic groups can be completed with a morphism~$\f$, thereby showing that every embedding problem for $\Gamma$ has a weak solution.  (We note that \cite{BassLubotzkyMagidMozes:TheProalgebraicCompletionOfRigidGroups} has the standing hypothesis that the base field $K$ is algebraically closed and of characteristic zero; but the self-contained proof of \cite[Prop.~4]{BassLubotzkyMagidMozes:TheProalgebraicCompletionOfRigidGroups} does not make use of that assumption.)
	\end{proof}
 
The following lemma will allow us to reduce the solvability of admissible embedding problems to the solvability of algebraic embedding problems and split admissible embedding problems.

\begin{lemma} \label{lem: domination}
	Let $\Gamma$ be a proalgebraic group. If every algebraic embedding problem for $\Gamma$ has a weak solution, then every admissible embedding problem $G\twoheadrightarrow H,\ \Gamma\twoheadrightarrow H$ for $\Gamma$ is dominated by a split admissible embedding problem.  That is, there exists a split admissible embedding problem $G'\twoheadrightarrow H',\ \Gamma\twoheadrightarrow H'$ and epimorphisms $G'\twoheadrightarrow G$, $H'\twoheadrightarrow H$ such that 
	\begin{equation} \label{eqn:diagram dominate}
	\xymatrix{
		&  \Gamma \ar@{->>}[d] \ar@/^2pc/@{->>}[dd] \\
		G' \ar@{->>}[r] \ar@{->>}[d] & H' \ar@{->>}[d] \\
		G \ar@{->>}[r] & H 	
	}
	\end{equation}
	commutes.	
\end{lemma}

\begin{proof} Let $G\twoheadrightarrow H,\ \Gamma\twoheadrightarrow H$ be an admissible embedding problem for $\Gamma$. In particular, $N:=\operatorname{ker}(G\twoheadrightarrow H)$ is algebraic. It follows from Lemma \ref{lemma: algebraic suffices for projective} that there exists a weak solution, i.e., a morphism $\f\colon\Gamma\to G$ such that
	$$\xymatrix{
	 \Gamma \ar@{->>}[rd] \ar_{\f}[d] & \\
		G \ar@{->>}[r]& H 
	}
	$$
	commutes. Let $H'=\f(\Gamma)\leq G$. The group $G$ (and therefore also $H'$) acts on the normal subgroup $N$ via conjugation and we can form the semidirect product $G'=N\rtimes H'$.
	So we have a split embedding problem $G'\twoheadrightarrow H'$ and $\Gamma\twoheadrightarrow H'$. 
	
	The rank of a closed subgroup of $G$ is at most the rank of $G$ (\cite[Lemma 3.5]{Wibmer:FreeProalgebraicGroups}), and the source and target of an epimorphism with algebraic kernel have the same rank (\cite[Lemma 3.6]{Wibmer:FreeProalgebraicGroups}. It follows that 
	$\rank(H')\leq \rank(G)=\rank(H)<\rank(\Gamma)$.	Moreover, $\ker(G'\twoheadrightarrow H')=N$ is algebraic. Thus the embedding problem  $G'\twoheadrightarrow H'$, $\Gamma\twoheadrightarrow H'$ is admissible.

	The restriction of $G\twoheadrightarrow H$ to $H'\to H$ is an epimorphism because $\Gamma\twoheadrightarrow H$ is an epimorphism.
	We will show that the multiplication map $G'\to G,\ (n,h')\mapsto nh'$ is also an epimorphism. 

	Let $R$ be a $K$-algebra and $g\in G(R)$. Let $h\in H(R)$ be the image of $g$ under $G\twoheadrightarrow H$. Since $\Gamma\twoheadrightarrow H$ is an epimorphism there exists a faithfully flat $R$-algebra $S$ and a $\gamma\in\Gamma(S)$ that maps onto $h\in H(R)\subseteq H(S)$. 
	Now $g\in G(R)\subseteq G(S)$ and $\f(\gamma)\in G(S)$ both map to $h\in H(S)$. Thus $n=g\f(\gamma)^{-1}\in N(S)$ and $(n,\f(\gamma))\in G'(S)$ maps onto $g\in G(S)$. Therefore $G'\to G$ is an epimorphism. By construction, diagram (\ref{eqn:diagram dominate}) commutes.
\end{proof}

\begin{prop} \label{prop: group theoretic reduction to algebraic and split}
	Let $\Gamma$ be a proalgebraic group. Then $\Gamma$ is saturated if and only if
\renewcommand{\theenumi}{(\roman{enumi})}
\renewcommand{\labelenumi}{(\roman{enumi})}  
	\begin{enumerate}
		\item\label{sat i}  every algebraic embedding problem for $\Gamma$ has a weak solution and
		\item\label{sat ii} every split admissible embedding problem for $\Gamma$ has a solution.
	\end{enumerate}
\end{prop}

\begin{proof}
	Clearly~\ref{sat i} and~\ref{sat ii} must be satisfied if $\Gamma$ is saturated.
	Let $G\twoheadrightarrow H,\ \Gamma\twoheadrightarrow H$ be an admissible embedding problem. By~\ref{sat i} and Lemma~\ref{lem: domination} there exists a split admissible embedding problem dominating it as in diagram (\ref{eqn:diagram dominate}). By~\ref{sat ii} this split admissible embedding problem has a solution $\Gamma\twoheadrightarrow G'$. Composing this solution with $G'\twoheadrightarrow G$ yields a solution to the original embedding problem.
\end{proof}

\section{Differential Galois theory} \label{sect: dgt}

In the first part of this section we  recall 
some basic results from differential Galois theory. We also introduce some notation that will be used throughout the text. In the second part, using Proposition \ref{prop: group theoretic reduction to algebraic and split} and \cite{BachmayrHarbaterHartmannWibmer:FreeDifferentialGaloisGroups}, we give a criterion for the freeness of the absolute differential Galois group in terms of differential embedding problems (Proposition~\ref{prop: criterion for abs differential group to be free}).

Throughout this paper all differential fields are assumed to be of characteristic zero.
The letter $F$ will always denote a differential field with derivation $\partial\colon F\to F$ and field of constants $K$. 
(The field of constants of a differential field $L$ is $L^\partial=\{a\in F|\ \partial(a)=0\}$.)  For now, we do not assume that $K$ is algebraically closed.

\subsection{Picard-Vessiot extensions and differential Galois groups} \label{subsec: PV dgg}

Classically, one considers Picard-Vessiot extensions $L/F$ associated to a given linear differential equation $\partial(y)=Ay$, where $A\in F^{n\times n}$. In this case $L$ is generated as a field by the entries of a fundamental solution matrix of $\partial(y)=Ay$. In particular, $L$ is a finitely generated field extension of $F$. (Standard references are \cite{SingerPut:differential} and \cite{Magid:LecturesOnDifferentialGaloisTheory}.) For our purposes it is essential to consider Picard-Vessiot extensions that are not finitely generated as field extensions.  Instead of considering a Picard-Vessiot extension of a single equation, we consider a Picard-Vessiot extension of a family of equations
\begin{equation} \label{eq: family} \partial(y)=A_iy,\ A_i\in F^{n_i\times n_i},\ i\in I.\end{equation}

Differential Galois theory in this generality is treated in \cite{AmanoMasuokaTakeuchi:HopfPVtheory} (the Hopf-algebraic definition there is equivalent to our definition of Picard-Vessiot extensions given 
below by \cite{Takeuchi:hopfalgebraicapproach}, Cor.~3.5.) 

\begin{defi}
	A differential field extension $L/F$ is a \emph{Picard-Vessiot extension} for (\ref{eq: family}) if
\renewcommand{\theenumi}{(\roman{enumi})}
\renewcommand{\labelenumi}{(\roman{enumi})}  
	\begin{enumerate}
		\item for every $i\in I$ there exists $Y_i\in\Gl_{n_i}(L)$ such that $\partial(Y_i)=A_iY_i$,
		\item $L$ is generated as a field extension of $F$ by the entries of all the $Y_i$'s,
		\item\label{PV iii} $L^\partial=F^\partial$.
	\end{enumerate}
	The differential $F$-subalgebra $R$ of $L$ generated by all the entries and the inverses of the determinants of all these $Y_i$'s is called a \emph{Picard-Vessiot ring} for (\ref{eq: family}). 	
\end{defi}

A Picard-Vessiot ring for (\ref{eq: family}) can also be characterized without reference to a Picard-Vessiot extension: It is a differential $F$-algebra $R$ such that
\begin{enumerate}
	\item for every $i\in I$ there exists $Y_i\in\Gl_{n_i}(R)$ such that $\partial(Y_i)=A_iY_i$,
	\item $R$ is generated as an $F$-algebra by all the entries of all the $Y_i$'s and the inverses of their determinants
	\item $R$ is $\partial$-simple, i.e., all $\partial$-ideals of $R$ are trivial and
	\item $R^\partial=F^\partial$. (If $K=F^\partial$ is algebraically closed, this condition follows automatically 
from the three other conditions.)
\end{enumerate}

A \emph{Picard-Vessiot extension} (resp. \emph{Picard-Vessiot ring}) is a Picard-Vessiot extension (resp. ring) with respect to some family of differential equations.

If $K=F^\partial$ is algebraically closed, there exists a Picard-Vessiot extension for any family of equations, and such a Picard-Vessiot extension or a Picard-Vessiot ring is unique up to an isomorphism of differential $F$-algebras. 
In that case, we can speak of ``the Picard-Vessiot extension'' 
associated to an equation, or to a family of equations. 
We note that \cite{AmanoMasuokaTakeuchi:HopfPVtheory} provides a characterization of Picard-Vessiot extensions and rings without reference to a family of differential equations.

\begin{defi} \label{def: dgg}
	The \emph{differential Galois group} $\gal(L/F)=\gal(R/F)$ of a Picard-Vessiot extension $L$ or a Picard-Vessiot ring $R$ 
 is the functor that associates to a $K$-algebra $T$ the group of differential automorphisms $\Aut^\del(R\otimes_K T/F\otimes_K T)$ of $R\otimes_K T$ over $F\otimes_K T$.
\end{defi}

Here $T$ is considered as a differential algebra with the trivial derivation. The differential Galois group is a proalgebraic group over $K$.
In fact $G=\gal(R/F)$ is represented by $K[G]=(R\otimes_F R)^\partial$ and the canonical map $R\otimes_KK[G]\to R\otimes_F R$ is an isomorphism. 
 In Section~\ref{sec: Solving differential embedding problems} this isomorphism will be interpreted geometrically by saying that $\spec(R)$ is a differential $G$-torsor.

\begin{rem} \label{rem: finite type}
	Let $L/F$ be a Picard-Vessiot extension with Picard-Vessiot ring $R$ and differential Galois group $G$. Then the following statements are equivalent: 
	\begin{enumerate}
		\item $L$ is a finitely generated field extension of $F$.
		\item $R$ is a finitely generated $F$-algebra.
		\item $G$ is algebraic (i.e., of finite type).
		\item $L/F$ is a Picard-Vessiot extension of a single equation.
	\end{enumerate}
\end{rem}

A Picard-Vessiot extension or ring satisfying the above equivalent conditions is said to be \emph{of finite type}.

Consider a Picard-Vessiot ring $R$ for a (possibly infinite) family of differential equations indexed over some set $I$. By viewing $I$ as the filtered direct limit of its finite subsets, we can realize $R$ as the filtered direct limit $\varinjlim_i R_i$ of Picard-Vessiot rings $R_i/F$ of finite type. Here each $R_i$ is in fact a Picard-Vessiot ring of a single equation, by the above remark.  If $R_i$ has differential Galois group $G_i$ and $T$ is a $K$-algebra, then an element $g\in G(T)=\Aut^\del(R\otimes_K T/F\otimes_K T)$ restricts to an element $g_i\in G_i(T)=\Aut^\del(R_i\otimes_K T/F\otimes_K T)$ for every $i$. Conversely, to define a $g\in \Aut^\del(R\otimes_K T/F\otimes_K T)$ is equivalent to defining a compatible system of $g_i$'s, i.e., $G(T)=\varprojlim_i G_i(T)$ and so $G=\varprojlim_i G_i$.

Let $L/F$ be a Picard-Vessiot extension with Picard-Vessiot ring $R$, $a\in L$, $T$ a $K$-algebra and $g\in G(T)$, where $G=\gal(R/F)$. The automorphism $g\colon R\otimes_K T\to R\otimes_K T$ extends to an automorphism $\widetilde{g}$ of the total ring of fractions of $R\otimes_K T$, which contains $L$. If $\widetilde{g}(a)=a$ we say that \emph{$a$ is fixed by $g$}. For a closed subgroup $H$ of $G$ we set 
$$L^H=\{a\in L|\ a \text{ is fixed by all } h\in H(T) \text{ for all $K$-algebras } T\}.$$
A similar notation will be used for Picard-Vessiot rings, i.e., 
$$R^H=\{a\in R|\ a \text{ is fixed by all } h\in H(T) \text{ for all $K$-algebras } T\}.$$

We can now recall the fundamental theorem of differential Galois theory, the {\em Galois correspondence}:
	Let $L/F$ be a Picard-Vessiot extension with differential Galois group $G$.
\renewcommand{\theenumi}{(\alph{enumi})}
\renewcommand{\labelenumi}{(\alph{enumi})}  	
	\begin{itemize}
		\item\label{diff corr a}  The assignments $M\mapsto\gal(L/M)$ and $H\mapsto L^H$ are mutually inverse bijections between the set all intermediate differential fields $F\subseteq M\subseteq L$ and the set of all closed subgroups $H$ of $G$.
		\item\label{diff corr b} For an intermediate differential field $M$, the extension $M/F$ is Picard-Vessiot if and only if the corresponding closed subgroup $H$ of $G$ is normal. Moreover, in this case, the restriction map $\gal(L/F)\to \gal(M/F)$ is an epimorphism with kernel $\gal(L/M)$, so that $\gal(M/F)\cong\gal(L/F)/\gal(L/M)$.
	\end{itemize}

If $R$ is the Picard-Vessiot ring of $L/F$, then in \ref{diff corr b} above, the Picard-Vessiot ring of $M/F$ is $R^H$.

\begin{defi}
	If $K=F^\partial$ is algebraically closed, we define the \emph{absolute differential Galois group of $F$} as the differential Galois group of the Picard-Vessiot extension for the family of all linear differential equations over $F$.
\end{defi}

\subsection{Differential embedding problems} 

In this section, we apply our result on embedding problems of proalgebraic groups (Proposition \ref{prop: group theoretic reduction to algebraic and split}) in the differential setting to obtain a new characterization of free absolute differential Galois groups in terms of differential embedding problems.

\begin{defi}
	A \emph{differential embedding problem} over $F$ is a pair $(\alpha:G\twoheadrightarrow H, R)$ where $R$ is a Picard-Vessiot ring over $F$ with $H\cong\gal(R/F)$,  and  $\alpha\colon G\twoheadrightarrow H$ is an epimorphism of proalgebraic groups. A \emph{(proper) solution} is an embedding of Picard-Vessiot rings $R\hookrightarrow S$ such that $G\cong\gal(S/F)$ and the diagram
	$$
	\xymatrix{
		\gal(S/F) \ar@{->>}^{\operatorname{res}}[r] & \gal(R/F) \\
		G \ar[u] \ar@{->>}[r] & H	\ar[u]
	}
	$$
	commutes. A differential embedding problem is \emph{split} if there exists a morphism $\alpha'\colon H\to G$ such that $\alpha\alpha'$ is the identity. In this case $G\cong N\rtimes H$ with $N=\ker(\alpha)$, and we may also refer to the pair $(N\rtimes H, R)$ as a differential embedding problem.
	
	A differential embedding problem is \emph{of finite type} if $G$ is an algebraic group, i.e., an affine group scheme of finite type. (This implies that $R/F$ is of finite type.)
\end{defi}

If $K=F^\partial$ is algebraically closed, one can use the differential Galois correspondence for the Picard-Vessiot extension of the family of all linear differential equations over $F$ to translate embedding problems for the absolute differential Galois group of $F$ into differential embedding problems over $F$ and vice versa (cf. \cite[Section~3.3]{BachmayrHarbaterHartmannWibmer:FreeDifferentialGaloisGroups}).

\begin{defi}
	The {\em rank} of a Picard-Vessiot extension $L/F$, denoted by $\rank(L)$, is the smallest cardinal number $\kappa$ such that $L$ is a Picard-Vessiot extension for a family of cardinality $\kappa$.
\end{defi}
In particular, $\rank(L)\leq |F|$ for every Picard-Vessiot extension $L/F$. If $L/F$ is a Picard-Vessiot extension with Picard-Vessiot ring $R$ we also set $\rank(R):=\rank(L)$.
Lemma 3.2 of \cite{BachmayrHarbaterHartmannWibmer:FreeDifferentialGaloisGroups} provides several equivalent characterizations of the rank of a Picard-Vessiot extension. In particular, if $R$ is a Picard-Vessiot ring that is not of finite type then the rank of $R$ agrees with its dimension as an $F$-vector space. Moreover, if $L/F$ is a Picard-Vessiot extension with differential Galois group $G$, then $\rank(L)=\rank(G)$ (\cite[Lemma~3.3]{BachmayrHarbaterHartmannWibmer:FreeDifferentialGaloisGroups}).

\begin{defi}
	A differential embedding problem $(\alpha\colon G\twoheadrightarrow H,\ R)$ over $F$ is \emph{admissible} if $\rank(R)<|F|$ and $\ker(\alpha)$ is an algebraic group.
\end{defi} 
Admissible differential embedding problems correspond to admissible embedding problems of proalgebraic groups as defined in Section~\ref{ebp}.
Note that every embedding problem of finite type (i.e. with $G$ algebraic) is admissible.

The criterion given by the following proposition will be a crucial ingredient to proving Matzat's conjecture.

\begin{prop} \label{prop: criterion for abs differential group to be free}
	Assume that $K=F^\partial$ is algebraically closed. Then the absolute differential Galois group of $F$ is the free proalgebraic group on a set of cardinality $|F|$ if and only if
\renewcommand{\theenumi}{(\roman{enumi})}
\renewcommand{\labelenumi}{(\roman{enumi})}	
	\begin{enumerate}
		\item\label{free i} every differential embedding problem over $F$ of finite type has a solution and
		\item\label{free ii} every split admissible differential embedding problem over $F$ has a solution.
	\end{enumerate}
\end{prop}

\begin{proof}
Let $\Gamma$ denote the absolute differential Galois group of $F$. It was shown in \cite[Theorem~3.7]{BachmayrHarbaterHartmannWibmer:FreeDifferentialGaloisGroups} that $\Gamma$ is the free proalgebraic group on a set of cardinality $|F|$ if and only if every admissible differential embedding problem over $F$ is solvable. Thus, if $\Gamma$ is the free proalgebraic group on a set of cardinality $|F|$, clearly~\ref{free i} and~\ref{free ii} must be satisfied.

To establish the converse, we first show that $\rank(\Gamma)=|F|$. Let $\widetilde{F}$ denote the Picard-Vessiot extension for the family of all linear differential equations over $F$, so that $\Gamma=\gal(\widetilde{F}/F)$. Fix a non-trivial algebraic group $G$ over $K$ and let $L$ denote the compositum of all Picard-Vessiot extensions of $F$ with differential Galois group isomorphic to $G$ (inside $\widetilde{F}$). If $\rank(L)=|F|$, then also $\rank(\Gamma)=|F|$ as claimed (using \cite[Lemma~3.3]{BachmayrHarbaterHartmannWibmer:FreeDifferentialGaloisGroups}). Assume for contradiction that $\rank(L)<|F|$. Let $R$ denote the Picard-Vessiot ring of $L$ and let $\alpha\colon G\times \gal(R/F)\to \gal(R/F)$ denote the projection onto the second factor. Then $(\alpha, R)$ is a split admissible differential embedding problem over~$F$. By~\ref{free ii} it has a solution, i.e., there exists a Picard-Vessiot ring $S$ over $F$ containing $R$ with $\gal(S/F)\cong G\times \gal(R/F)$ such that 
	$$
\xymatrix{
 G\times \gal(R/F)	 \ar_\simeq[d] \ar@{->>}[rd] &	 \\
	\gal(S/F) \ar@{->>}[r] & \gal(R/F) 
}
$$
commutes. Since $K$ is algebraically closed, any Picard-Vessiot ring over $F$ embeds into $\widetilde{F}$, so we may assume that $S$ is contained in $\widetilde{F}$. The Picard-Vessiot extension $L'/F$ corresponding to the closed normal subgroup $\gal(R/F)$ of $G\times\gal(R/F)\simeq \gal(S/R)$ has differential Galois group isomorphic to $G$ but is not contained in $L$ by the Galois correspondence (as $L$ and $L'$ correspond to the subgroups $G$ and $\gal(R/F)$ of $G\times \gal(R/F)$). 
This contradicts the choice of $L$. Thus $\rank(\Gamma)=|F|\geq |K|$.

On the other hand, we can use Proposition \ref{prop: group theoretic reduction to algebraic and split} to deduce from~\ref{free i} and~\ref{free ii} that $\Gamma$ is saturated. But a proalgebraic group $\Gamma$ over $K$ with $\rank(\Gamma)\geq|K|$ is saturated if and only if it is free on a set of cardinality $\rank(\Gamma)$ (\cite[Theorem 3.42]{Wibmer:FreeProalgebraicGroups}). Thus $\Gamma$ is free on a set of cardinality~$|F|$.
\end{proof}

\section{Solving differential embedding problems via patching} 
\label{sec: Solving differential embedding problems}
Proposition \ref{prop: criterion for abs differential group to be free} leaves us with the task to solve all differential embedding problems of type~\ref{free i} and~\ref{free ii} as in that proposition. All differential embedding problems over $\C(x)$ of type~\ref{free i} were already solved in \cite[Cor. 4.6]{BachmayrHartmannHarbaterPopLarge}, so it remains to consider split admissible differential embedding problems.

The main goal of this section is to establish a certain patching setup that yields solutions to split admissible differential embedding problems (Theorem \ref{theo: main patching}). This theorem generalizes \cite[Theorem~2.14]{BachmayrHarbaterHartmannWibmer:DifferentialEmbeddingProblems} from Picard-Vessiot rings of finite type to arbitrary Picard-Vessiot rings. The main strategy will be to write a proalgebraic group as the projective limit of algebraic groups and provide compatible solutions of the corresponding differential embedding problems of finite type. 

As before, $F$ always denotes a differential field of characteristic zero with field of constants $K$.

\subsection{Reduction to the finite type case}

It was noted in Section~\ref{subsec: PV dgg} that every Picard-Vessiot ring can be viewed as a direct limit of Picard-Vessiot rings of finite type, and that the differential Galois group is the projective limit of the corresponding algebraic differential Galois groups. The next lemma shows the converse, for arbitrary directed systems of Picard-Vessiot rings.

\begin{lem} \label{lem: direct system of PV rings}
	Let $(R_i)_{i\in I}$ be a directed system of Picard-Vessiot rings over $F$. Then $R=\varinjlim R_i$ is a Picard-Vessiot ring over $F$ and $\gal(R/F)=\varprojlim\gal(R_i/F)$.
\end{lem}
\begin{proof}
	Since Picard-Vessiot rings are differentially simple, any morphism of Picard-Vessiot rings is injective. We can thus identify each $R_i$ with a subring of $R$. Then $R$ is the union of the $R_i$'s and similarly the field of fractions of $R$ is the union of the fields of fractions of the $R_i$'s. As there are no new constants in the latter differential fields, there are also no new constants in the fraction field of $R$. Thus, if $R_i$ is a Picard-Vessiot ring for a familiy $\mathcal{F}_i$ of linear differential equations over $F$, then $R$ is a Picard-Vessiot ring for $\cup \mathcal{F}_i$.
	
	Let $T$ be an algebra over $K=F^\del$. For a given $g\in \gal(R/F)(T)=\Aut^\del(R\otimes_K T/F\otimes_K T)$, the restriction maps $\gal(R/F)\to \gal(R_i/F)$ define an element $(g_i)$ of $(\varprojlim \gal(R_i/F))(T)=\varprojlim \gal(R_i/F)(T)$. Conversely, given an element $(g_i)$ of $\varprojlim \gal(R_i/F)(T)$, we can define an element $g$ of $\gal(R/F)(T)$ by setting $g(a)=g_i(a)$ for $a\in R_i\otimes_K T\subseteq R\otimes_K T$.
\end{proof}

Recall that a coalgebraic subgroup of a proalgebraic group $G$ is a normal closed subgroup $U\trianglelefteq G$ such that $G/U$ is algebraic, i.e., of finite type. 
The intersection of two coalgebraic subgroups $U,V$ of $G$ is also coalgebraic, since $G/(U\cap V)$ embeds into to the algebraic group $G/U\times G/V$.
Thus the set of coalgebraic subgroups of $G$ is a directed set, by defining $U\preccurlyeq V$ if and only if $U\supseteq V$. If $\mathcal{U}$ is a directed family of coalgebraic subgroups of $G$ such that $\bigcap_{U\in\mathcal{U}} U=1$, then $G=\varprojlim_\mathcal{U} G/U$.

If moreover $R$ is a Picard-Vessiot ring with differential Galois group $G$, then, by the Galois correspondence, $R=\varinjlim_\mathcal{U} R^U$, and $R^U$ is a Picard-Vessiot ring of finite type with differential Galois group $G/U$.

The following lemma will allow us to break down a split admissible differential embedding problem $(N\rtimes H, R)$ to a family of split differential embedding problems $(N\rtimes (H/U), R^U)$ of finite type. 

\begin{lemma} \label{lemma: safe}
	Let $H$ be a proalgebraic group acting (from the left) on an algebraic group $N$. Then there exists a directed family $\mathcal{U}$ of coalgebraic subgroups of $H$ with $\bigcap_{U\in\mathcal{U}} U=1$ such that each $U\in \mathcal{U}$ acts trivially on $N$. For the induced actions $H/U\times N\to N$ we then have a commutative diagram
	$$
	\xymatrix{
		H\times N \ar[d] \ar[rd] & \\
		H/V\times N \ar[r] \ar[d] & N \\
		H/U\times N 	\ar[ru]
	}
	$$ 
	for all $U\supseteq V$ in $\mathcal{U}$.
\end{lemma}

\begin{proof}
	Let $\rho\colon K[N]\to K[H]\otimes_K K[N]$ denote the co-action associated to the action $H \times N \to N$ of $H$ on $N$. Since $K[N]$ is a finitely generated $K$-algebra, there exists a finitely generated $K$-subalgebra $T$ of $K[H]$ such that $\rho(K[N])\subseteq T\otimes_K K[N]$. This $T$ is contained in a finitely generated Hopf-subalgebra of $K[H]$ (\cite[Theorem 3.3]{Waterhouse:IntroductiontoAffineGroupSchemes}) and this Hopf-subalgebra is of the form $K[H/U_0]$ for some coalgebraic subgroup $U_0$ of $H$. So the co-action $\rho$ factors through $K[H/U_0]\otimes_K K[N]\subseteq K[H]\otimes_K K[N]$.  The action of $H$ on $N$ thus factors through $H/U_0$, and the action of $U_0 \subseteq H$ on $N$ is trivial.  These properties also hold for any coalgebraic subgroup $U$ of $H$ contained in $U_0$. Consider the family $\mathcal{U}$ of all such $U$'s. The intersection of all coalgebraic subgroups of $H$ is trivial since $H$ is a proalgebraic group. Moreover, for every coalgebraic subgroup $V$ of $H$, the intersection $U_0\cap V$ is also a coalgebraic subgroup of~$H$. Therefore, the intersection of all $U$'s is trivial and the family $\mathcal{U}$ is as asserted.
\end{proof}

\begin{lemma} \label{lemma: reduce split to finite type}
	Let $(N\rtimes H, R)$ be a split differential embedding problem with $N$ algebraic and let $\mathcal{U}$ be a family of coalgebraic subgroups of $H$ as in Lemma \ref{lemma: safe}. Assume that for every $U\in \mathcal{U}$ a solution $S_U$ of the differential embedding problem $(N\rtimes (H/U), R^U)$ is given, together with morphisms $f_{UV}\colon S_U\to S_V$ for $U\supseteq V$ subject to the following conditions:
\renewcommand{\theenumi}{(\roman{enumi})}
\renewcommand{\labelenumi}{(\roman{enumi})}	
	\begin{enumerate}
		\item\label{family i}  $f_{VW}\circ f_{UV}=f_{UW}$ and $f_{UU}=\id$ for all $U\supseteq V\supseteq W$,
		\item\label{family ii} $$\xymatrix{ R^U \ar@{^(->}[r] \ar@{^(->}[d] & S_U \ar@{^(->}[d]\\
			R^V \ar@{^(->}[r] & S_V
		}$$
		commutes for all $U\supseteq V$ and
		\item\label{family iii} $$\xymatrix{
			\gal(S_V/F)  \ar[r]\ar@{->>}[d] & N\rtimes (H/V)\ar@{->>}[d] \\
			\gal(S_U/F)  \ar[r] &\ N\rtimes(H/U)
		}$$ 	commutes for all $U\supseteq V$.
	\end{enumerate}
	Then $S=\varinjlim S_U$ is a solution of the differential embedding problem  $(N\rtimes H, R)$.	
\end{lemma}
\begin{proof}
It follows from Lemma \ref{lem: direct system of PV rings} that $S$ is a Picard-Vessiot ring with differential Galois group $\gal(S/F)=\varprojlim\gal(S_U/F)$, with respect to the inverse system defined by~\ref{family i}.
Condition~\ref{family iii} gives an isomorphism between the inverse systems $(\gal(S_U/F))_U$ and $(N\rtimes H/U)_U$. Therefore $\gal(S/F)=\varprojlim\gal(S_U/F)\simeq \varprojlim N\rtimes H/U=N\rtimes H.$
Using~\ref{family ii} we obtain an embedding $R=\varinjlim R^U\hookrightarrow \varinjlim S_U=S$. The commutativity of
	$$
	\xymatrix{
		\gal(S/F) \ar@{->>}[r] \ar@{->>}[d] & \gal(S_U/F) \ar^-\simeq [r] \ar@{->>}[d] & N\rtimes H/U \ar@{->>}[d] \\
		\gal(R/F) \ar@{->>}[r] & \gal(R^U/F) \ar^-\simeq [r]  & H/U	 }
	$$
	implies the commutativity of 
	$$
	\xymatrix{
		\gal(S/F) \ar@{->>}[d] \ar^\simeq[r] & N\rtimes H \ar@{->>}[d] \\
		\gal(R/F) \ar^-\simeq[r]  & H.	
	}
	$$
	Thus $S$ is a solution of $(N\rtimes H, R)$.
\end{proof}

\subsection{Patching problems of differential torsors }\label{sec patching}

In this section, $G$ is an algebraic group over $K$. As always, $F$ is a differential field of characteristic zero with field of constants $K$.

Recall that an \textit{affine $G$-space} is an affine scheme $X=\Spec(R)$ of finite type over $K$ equipped with a morphism action $\alpha \colon X\times G \to X$ such that $\alpha_T\colon X(T)\times G(T)\to X(T)$ defines a right group action of $G(T)$ on $X(T)$ for every $K$-algebra $T$; the dual of $\alpha$ is the co-action $\rho\colon K[X]\to K[X]\otimes_K K[G]$. A non-empty affine $G$-space $X$ is called \textit{$G$-torsor} if for every $K$-algebra $T$ and all $x,y \in X(T)$ there exists a unique $g\in G(T)$ with $x.g=y$, where we write $\alpha_T(x,g)=x.g$.

A $G_F$-torsor $X$ equipped with an extension of the derivation from $F$ to $F[X]$ is called a \textit{differential $G_F$-torsor} if the co-action $\rho\colon F[X]\to F[X]\otimes_F F[G]$ is compatible with the derivation, i.e., is a differential homomorphism (here we consider $F[G]=F\otimes_K K[G]$ as a differential ring extension of $F$ with constants $K[G]$). Differential torsors were introduced in \cite{BachmayrHarbaterHartmannWibmer:DifferentialEmbeddingProblems} for the purpose of solving differential embedding problems of finite type. 
It is possible to define differential torsors in the more general context where $G$ and $X$ are not necessarily of finite type of $K$. However, to be able to apply the results from \cite{BachmayrHarbaterHartmannWibmer:DifferentialEmbeddingProblems}, we restrict ourselves to the finite type situation.

If $R$ is a Picard-Vessiot ring over $F$ with differential Galois group $G$, then $X=\spec(R)$ is a differential $G_F$-torsor. Conversely, if $X=\Spec(R)$ is a differential $G_F$-torsor such that $R$ is a simple differential ring with field of constants $K$, then $R$ is a Picard-Vessiot ring over $F$ with differential Galois group $G$ by \cite[Prop. 1.12]{BachmayrHarbaterHartmannWibmer:DifferentialEmbeddingProblems}. 

To consider patching problems of differential torsors, we work with \textit{differential diamonds} $(F,F_1,F_2,$ $F_0)$, that is, overfields $F_1$ and $F_2$ of $F$ that are contained in a common overfield $F_0$ with compatible extensions of the derivation on $F$ such that $F_1\cap F_2=F$. We say that such a diamond has the \textit{factorization property} if for every matrix $A\in \GL_n(F_0)$ there exist matrices $B\in \GL_n(F_1)$ and $C \in \GL_n(F_2)$ with $A=BC$.  

  Let $(F,F_1,F_2,F_0)$ be a differential diamond.  
Recall that a {\em patching problem} of differential $G$-torsors in this situation 
is a tuple $(X_1,X_2,X_0,\nu_1,\nu_2)$ where $X_i = \Spec(S_i)$ is a differential $G_{F_i}$-torsor, and where $\nu_i:F_0 \times_{F_i} X_i \to X_0$ is an isomorphism of differential $G_{F_0}$-torsors for $i=1,2$.  
On the level of coordinate rings, this patching problem corresponds to the tuple $(S_1,S_2,S_0,\Theta_1,\Theta_2)$, where $\Theta_i = (\nu_i^*)^{-1}:F_0 \otimes_{F_i}S_i\to S_0$ is an isomorphism of differential $F_0$-algebras that respects the $G$-co-action.  
We also refer to this tuple as a patching problem of differential $G$-torsors.  A {\em solution} to the patching problem is a differential torsor over $F$ that induces the torsors $X_i$ compatibly via base change; i.e., a differential $G$-torsor $X = \Spec(S)$ over $F$ 
together with $F_i$-isomorphisms of $G_{F_i}$-torsors $\gamma_i: F_i \times_F X \to X_i$ for $i=1,2$
such that the two maps $\nu_i \circ (\mathrm{id}_{F_0}\otimes_{F_i} \gamma_i): F_0 \times_F X \to X_0$ agree.  On the level of coordinate rings, we can write $\Phi_i = (\gamma_i^{-1})^*:F_i \otimes_F S \to S_i$, and the compatibility condition is that 
\[\Theta_1\circ(\operatorname{id}_{F_0}\otimes_{F_1}\Phi_1)=
\Theta_2\circ(\operatorname{id}_{F_0}\otimes_{F_2}\Phi_2).\]
We then also refer to $(S,\Phi_1,\Phi_2)$ as a solution to the patching problem.

Building on a result in \cite{HarbaterHartmannKrashen:LocalGlobalPrinciplesForTorsors} on patching problems of $G$-torsors (without differential structure) it was shown in \cite[Thm. 2.2]{BachmayrHarbaterHartmannWibmer:DifferentialEmbeddingProblems} that every patching problem of differential $G$-torsors has a solution if $(F,F_1,F_2,F_0)$ has the factorization property. Moreover, this solution $(S,\Phi_1,\Phi_2)$ is unique up to differential isomorphism and it is given by $S=\Theta_1(S_1)\cap\Theta_2(S_2)\subseteq S_0$, with derivation and $G$-action given by restriction from those on $\Theta_1(S_1)$ and $\Theta_2(S_2)$ and with $\Phi_i$ induced by $\Theta_i^{-1}$.

\subsection{Solving split differential embedding problems}

Recall that if $H$ is a closed subgroup of an algebraic group $G$ over a field $K$, and if $Y$ is an $H$-torsor over $K$, then there is a natural {\em induced torsor} $X=\Ind_H^G(Y)$ that satisfies a universal mapping property and is given by $(Y \times G)/H$, where $H$ acts on $Y \times G$ by $(y,g).h = (y. h,h^{-1}g)$.  On the level of coordinate rings, if $Y = \Spec(R)$, then
\[\Ind_H^G(R) := K[X]=(K[Y]\otimes_K K[G])^H=\{ f \in K[Y]\otimes_K K[G] \mid \rho(f)=f\otimes 1\},\] 
where $\rho: K[Y]\otimes_K K[G] \to K[Y]\otimes_K K[G]\otimes_K K[H]$ is the co-action corresponding to the action of $H$ on $Y\times G$. 
See \cite[Proposition~A.8]{BachmayrHarbaterHartmannWibmer:DifferentialEmbeddingProblems} and its proof for details.  Another interpretation of induced torsors is as follows: the set of isomorphism classes of $G$-torsors over $K$ is classified by the Galois cohomology set $H^1(K,G)$, and similarly for $H$-torsors.  If $Y$ is an $H$-torsor over $K$, corresponding to $\alpha \in H^1(K,H)$, then $\Ind_H^G(Y)$ is the $G$-torsor corresponding to the image of $\alpha$ under the natural map $H^1(K,H) \to H^1(K,G)$.

In the differential context, i.e., if $R$ is the coordinate ring of a differential $H_F$-torsor, then
$$\Ind_{H_F}^{G_F}(R)=(R\otimes_F F[G])^{H_F}$$
is the coordinate ring of a differential $G_F$-torsor. The derivation on $(R\otimes_F F[G])^{H_F}$ is the restriction of the derivation on $R\otimes_F F[G]$.

\begin{theo}\label{theo: main patching}
	Let $(F,F_1,F_2,F_0)$ be a differential diamond with the factorization property and define $K=F^\del$. 
	Let $G=N\rtimes H$ be a proalgebraic group over $K$ such that $N$ is algebraic. Assume further that $R$ is a Picard-Vessiot ring over $F$ with differential Galois group $H$ and such that $R\subseteq F_1$ as a differential $F$-subalgebra. Finally, we assume that there exists a Picard-Vessiot ring $R_1/F_1$ with differential Galois group $N_{{F_1}^{\! \del}}$ and with $R_1\subseteq F_0$ as a differential $F_1$-subalgebra.
	
	Then the differential embedding problem $(N\rtimes H, R)$ over $F$ has a solution.
\end{theo}

\begin{proof}
The strategy of our proof is to use Lemmas~\ref{lemma: safe} and~\ref{lemma: reduce split to finite type} to reduce to the case of finite type, where the result was shown in \cite[Prop. 2.11]{BachmayrHarbaterHartmannWibmer:DifferentialEmbeddingProblems}.  To be able to use Lemma~\ref{lemma: reduce split to finite type}, we will need to verify the compatibility conditions~\ref{family i}, \ref{family ii}, \ref{family iii} there; and for that we will need to know the explicit form of the solutions $S_U$ to the finite type problems.  We therefore first recall the proof of the assertion in the finite type case, as given in \cite{BachmayrHarbaterHartmannWibmer:DifferentialEmbeddingProblems}.

So assume that $H$ (and hence $G$) is an algebraic group.
The tuple $$(S_1, S_2, S_0, \Theta_1, \Theta_2)=\left(\Ind_{N_{F_1}}^{G_{F_1}}(R_1),\  F_2\otimes_F\Ind_{H_F}^{G_F}(R),\ F_0[G],\ \Theta_1,\ \Theta_2\right)$$
is a patching problem of differential $G$-torsors, where
$$\Theta_1\colon F_0\otimes_{F_1}S_1=F_0\otimes_{F_1} (R_1\otimes_{F_1} F_1[G])^{N_{F_1}}
\longrightarrow^{\!\!\!\!\!\!\!\!\!\!\sim\,}\ F_0[G]$$
is the restriction of $F_0\otimes_{F_1}R_1\otimes_{F_1} F_1[G]\to F_0[G],\ a\otimes r\otimes f\mapsto arf$ and
$$\Theta_2\colon F_0\otimes_{F_2}S_2=F_0\otimes_F(R\otimes_F F[G])^{H_F}
\longrightarrow^{\!\!\!\!\!\!\!\!\!\!\sim\,}\ F_0[G]$$
is the restriction of $F_0\otimes_F R\otimes_F F[G]\longrightarrow F_0[G], \ a\otimes r\otimes f\mapsto arf.$
The solution $S=\Theta_1(S_1)\cap \Theta_2(S_2)$ of this patching problem is a solution of the differential embedding problem $(N\rtimes H, R)$. In more detail: 
The derivation on $S$ is the restriction of the derivation on $F_0[G]$, and the co-action $S\to S\otimes_F F[G]$ is the restriction of the comultiplication $F_0[G]\to F_0[G]\otimes_{F_0} F_0[G]$; these are well-defined on $S$ because they are the common restrictions of the corresponding maps for 
$\Theta_1(S_1)$ and $\Theta_2(S_2)$.
The inclusion $R\to S$ is the restriction of $R\to R\otimes_F F[H]\to F_0[G]$, where the first map is the co-action and the second map is $r\otimes f\mapsto rf$. (To verify the latter one has to trace through the explicit form of the middle vertical map in the diagram in the proof of \cite[Prop. 2.11]{BachmayrHarbaterHartmannWibmer:DifferentialEmbeddingProblems}.)

We now return to the proof in the general situation, where $H$ need not be of finite type. 
By applying Lemma~\ref{lemma: safe} to the action of $H$ on $N$, we obtain a family 
$\mathcal{U}$ of coalgebraic subgroups of $H$ with the properties asserted there. Since each $U\in\mathcal{U}$ acts trivially on $N$, we can regard $U$ as a normal closed subgroup of $G$, and we can form the quotient $G/U = N \rtimes (H/U)$ of $G = N \rtimes H$. For every $U\in\mathcal{U}$ we have a differential embedding problem $(G/U, R^U)=(N\rtimes(H/U),R^U)$ of finite type with a solution $S_U$ defined (as above) as the solution of the patching problem
$$(S_{1U},S_{2U},S_{0U}, \Theta_{1U},\ \Theta_{2U})= \left(\Ind_{N_{F_1}}^{(G/U)_{F_1}}(R_1),\  F_2\otimes_F\Ind_{(H/U)_F}^{(G/U)_F}(R^U),\ F_0[G/U],\ \Theta_{1U},\ \Theta_{2U}\right).$$
That is, $S_U=\Theta_{1U}(S_{1U})\cap \Theta_{2U}(S_{2U})$ with $\Theta_{1U}$ and $\Theta_{2U}$ defined like $\Theta_1$ and $\Theta_2$ above.
If $U\supseteq V$ are from $\mathcal{U}$, then $R^U\subseteq R^V$
and $F_i[G/U]\subseteq F_i[G/V]$ for $i=0,1,2$. Therefore $$S_{1U}=(R_1\otimes_{F_1}F_1[G/U])^{N_{F_1}}\subseteq (R_1\otimes_{F_1}F_1[G/V])^{N_{F_1}}=S_{1V}$$ 
and similarly, since
$$
\xymatrix{
R^U\otimes_F F[G/U] \ar^-{\rho}[r] \ar@{^(->}[d] & R^U\otimes_F F[G/U]\otimes_F F[H/U] \ar@{^(->}[d] \\
R^V\otimes_F F[G/V]  \ar^-{\rho}[r] & R^V\otimes_F F[G/V]\otimes_F F[H/V]
}
$$
commutes, we have $$\Ind_{(H/U)_F}^{(G/U)_F}(R^U)=(R^U\otimes_F F[G/U])^{(H/U)_F}\subseteq (R^V\otimes_F F[G/V])^{(H/V)_F}=\Ind_{(H/V)_F}^{(G/V)_F}(R^V).$$
So $S_{iU}\subseteq S_{iV}$ for $i=0,1,2$. Moreover, 
\begin{equation} \label{eqn: commutei}
\xymatrix{
F_0\otimes_{F_1}S_{iU} \ar^-{\Theta_{iU}}[r] \ar[d] & S_{0U} \ar[d] \\
F_0\otimes_{F_1} S_{iV} \ar^-{\Theta_{iV}}[r] & S_{0V} 
}
\end{equation}
commutes for $i=1,2$.  Hence $S_{0U}\subseteq S_{0V}$ restricts to an inclusion $f_{UV}\colon S_U\to S_V$ of differential $F$-algebras.  

We claim that the solutions $S_U$ of the differential embedding problems $(N\rtimes(H/U),R^U)$, together with the morphisms $f_{UV}\colon S_U\to S_V$, satisfy the three conditions in the hypothesis of Lemma~\ref{lemma: reduce split to finite type}.  Once this is shown, the conclusion of that lemma yields the desired solution to the differential embedding problem $(N\rtimes H, R)$ over $F$.  

For condition~\ref{family i} of Lemma~\ref{lemma: reduce split to finite type}, note that these compatibilities hold for the corresponding morphisms between the associated patching problems, by the functoriality of induced torsors.  By the equivalence of categories between patching problems and solutions, it follows that condition~\ref{family i} holds as well for the morphisms $f_{UV}$.

Concerning condition~\ref{family ii}, consider the compositions
$R^U \hookrightarrow R^U\otimes_F F[H/U] \hookrightarrow F_0[G/U] = S_{0U}$
for all $U\in\mathcal{U}$.  For $U \supseteq V$, this yields a commutative 
square 
\begin{equation} \label{eqn: commuteii}
\xymatrix{
R^U \ar@{^(->}[r] \ar@{^(->}[d] & S_{0U} \ar@{^(->}[d] \\
R^V \ar@{^(->}[r]  & S_{0V}.
}
\end{equation}
Here the horizontal maps factor through $S_U$ and $S_V$ respectively, yielding a diagram
\[
\xymatrix{
R^U \ar@{^(->}[r] \ar@{^(->}[d] & S_U \ar@{^(->}[r] \ar@{^(->}[d] & S_{0U} \ar@{^(->}[d] \\
R^V \ar@{^(->}[r] & S_V \ar@{^(->}[r] & S_{0V}.
}
\]
The right hand square of this diagram commutes, because $\Theta_{iV}$ restricts to $\Theta_{iU}$ for $i=1,2$.
Since the maps in the above diagram are injective, and since diagram~(\ref{eqn: commuteii}) commutes, it follows that the left hand square in the above diagram also commutes; i.e., condition~\ref{family ii} holds.

For condition~\ref{family iii}, note that the comultiplication maps induce a commutative square 
\[
\xymatrix{
F_0[G/U] \ar[r] \ar[d] & F_0[G/U]\otimes_{F_0} F_0[G/U] \ar[d] \\
F_0[G/V] \ar[r] & F_0[G/V]\otimes_{F_0} F_0[G/V],	
}
\]
for $U \supseteq V$.  This implies the commutativity of 
\[
\xymatrix{
S_U \ar[r] \ar[d] & S_U\otimes_F F[G/U] \ar[d] \\
S_V \ar[r] & S_V\otimes_F F[G/V].	
}
\]
So the structure of $\Spec(S_V)$ as a differential $G/V$-torsor induces that of 
$\Spec(S_U)$ as a differential $G/U$-torsor.  This yields condition~\ref{family iii}, completing the proof of the claim and hence the theorem.
\end{proof}

\section{Matzat's conjecture} \label{sec: Matzat}

In this section, we conclude our proof that every split admissible differential embedding problem over 
$\C(x)$ has a solution, and thereby prove our main result (Theorem~\ref{result}).
Our strategy is first to show that the induced embedding problem over $\C((t))(x)$ has a solution (Theorem \ref{theo: solve F/E DEB}), and then to descend this solutions in a suitable way. 

In the proof of Theorem \ref{theo: solve F/E DEB}, we will work with differential diamonds of fields $(F,F_1,F_2,F_0)$ with $F_0$ a certain iterated Laurent series field. The following lemma will allow us to work inside such Laurent series fields.

\begin{lemma} \label{lemma: embed PV ring}
	Let $k$ be an algebraically closed field of characteristic zero and $E=k(x)$. If $R/E$ is a Picard-Vessiot ring with $\rank(R)<|k|$, then there exists an $a\in k$ such that $R$ embeds into $k((x-a))$ as a differential $k(x)$-algebra. 
\end{lemma}

\begin{proof}
	This statement is well known for Picard-Vessiot rings of finite type. So we may assume that $\rank(R)$ is infinite.
	
	We know that $R$ is the Picard-Vessiot ring for a family $(\partial(y)=A_i y)_{i\in I}$ with $|I|=\rank(R)<|k|$. The set of all poles of all coefficients of all $A_i$'s has cardinality $|I|<|k|$. Thus there exists an $a\in k$ such that all equations $\partial(y)=A_iy$ are regular at $a$ and therefore have a fundamental solution matrix $Y_i\in\Gl_{n_i}(k((x-a)))$. The $E$-subalgebra $R'$ of $k((x-a))$ generated by all $Y_i$'s and the inverse of their determinants is a Picard-Vessiot ring for $(\partial(y)=A_i y)_{i\in I}$. By the uniqueness of Picard-Vessiot rings, $R\cong R'$.
\end{proof}

The following lemma generalizes \cite[Prop. 2.4]{BachmayrHartmannHarbaterPopLarge} from Picard-Vessiot rings of finite type to arbitrary Picard-Vessiot rings.

\begin{lemma} \label{lemma: extend constants for PV ring}
	Let $E$ be a differential field and $K$ a field extension of $k=E^\partial$. Then 
$E\otimes_k K$ (with trivial derivation on $K$) is an integral domain whose field of fractions	
$F$ is a differential field with $F^\partial=K$. Moreover, if $R$ is a Picard-Vessiot ring over $E$ with differential Galois group $G$, then $R'=R\otimes_E F$ is a Picard-Vessiot ring over $F$ with differential Galois group $G'=G_{K}$.
\end{lemma}

\begin{proof}
	We first note that $E\otimes_k K$ is an integral domain because $k$ is relatively algebraically closed in $E$. 
	If $S$ is a $\partial$-simple ring with field of constants $k$, then $S\otimes_k K$ is $\partial$-simple with field of constants $K$ (\cite[Lemma 2.3]{DiVizioHardouinWibmer:DifferenceGaloisTheoryOfLinearDifferentialEquations}). 
Here $F^\partial=K$,	
since the field of constants of a $\partial$-simple ring agrees with the field of constants of its field of fractions (e.g., see the first part of the proof of \cite[Lemma 1.17(2)]{SingerPut:differential}).  As $R\otimes_EF$ is a localization of 
$R\otimes_E (E\otimes_k K) = 
R\otimes_kK$ which is $\partial$-simple, we see that $R\otimes_E F$ is also $\partial$-simple and has field of constants $K$. Thus $R'$ is a Picard-Vessiot ring over $F$.
	
For a $K$-algebra $T$ we have
	\begin{align*}
	G_{K}(T)&=\Aut^\partial(R\otimes_kT/E\otimes_k T)=\Aut^\partial((R\otimes_k K)\otimes_{K}T/(E\otimes_k K)\otimes_{K} T)=\\
	&=\Aut^\partial(R'\otimes_{K}T/F\otimes_{K} T)=G'(T)
	\end{align*}
and so $G'=G_{K}$.
\end{proof}

\begin{defi}
	Let $E$ be a differential field, let $K$ be a field extension of $k=E^\partial$,
and let $F$ be the fraction field of $E\otimes_k K$. If $(\alpha\colon G\twoheadrightarrow H,\ R)$ is a differential embedding problem over $E$, then $(\alpha_{K}\colon G_{K}\twoheadrightarrow H_{K},\ R\otimes_E F)$ is a differential embedding problem over $F$, by Lemma \ref{lemma: extend constants for PV ring}. Such an embedding problem is called an \emph{$(F/E)$-differential embedding problem}. \end{defi}

In the above definition, if the differential embedding problem $(\alpha\colon G\twoheadrightarrow H,\ R)$ is split (resp.\ admissible), then so is $(\alpha_{K}\colon G_{K}\twoheadrightarrow H_{K},\ R\otimes_E F)$.  We then refer to the latter embedding problem as an \emph{$(F/E)$-split} (resp.\ \emph{$(F/E)$-admissible}) differential embedding problem.  If $(\alpha\colon G\twoheadrightarrow H,\ R)$ has both properties, we call 
$(\alpha_{K}\colon G_{K}\twoheadrightarrow H_{K},\ R\otimes_E F)$ an \emph{$(F/E)$-split admissible} differential embedding problem.

The proof of the next theorem is modeled on the proof of \cite[Theorem~4.2]{BachmayrHarbaterHartmann:DifferentialEmbeddingProblemsOverLaurentSeriesFields}, replacing certain ingredients there by generalizations to the case of proalgebraic groups respectively not necessarily finitely generated Picard-Vessiot extensions. However, note that since we work over an algebraically closed base field $k$ (whereas \cite{BachmayrHarbaterHartmann:DifferentialEmbeddingProblemsOverLaurentSeriesFields} considered arbitrary fields $k$), some parts of the proof simplify significantly. 
\begin{theo} \label{theo: solve F/E DEB}
	Let $k$ be an algebraically closed field of characteristic zero, and let $K=k((t))$. Then every $(K(x)/k(x))$-split admissible differential embedding problem has a solution.
\end{theo}

\begin{proof} Let $(N\rtimes H, R)$ be a $(K(x)/k(x))$-split admissible differential embedding problem. 
	Thus $R$ is a Picard-Vessiot ring over $F:=K(x)$ with differential Galois group $H$.  Moreover,
	there exist proalgebraic groups $N_0$ and $H_0$ over $k$ with $N_0$ of finite type such that $N$ is the base change of $N_0$ and $H$ is the base change of $H_0$ from $k$ to $K$, and there exists a Picard-Vessiot ring $R_0$ over $E:=k(x)$ with differential Galois group $H_0$ and $\rank(R_0)<|k(x)|=|k|$ such that $R\cong R_0\otimes_{k(x)}K(x)$ as differential $K(x)$-algebras. 
	
	By Lemma~\ref{lemma: embed PV ring}, there exists an $a\in k$ such that $R_0$ embeds into $k((x-a))$ as a differential $k(x)$-algebra, 
	where $k((x-a))$ is equipped with the derivation $\frac{\del}{\del(x-a)}$. 
	We use the method of patching over fields (see \cite{HarbaterHartmann:PatchingOverFields}) over the $x$-line $\mathbb P^1_{k[[t]]}$. Let $P$ denote the point $(x-a,t)$ on the closed fibre $\mathbb P^1_{k}$ of $\mathbb P^1_{k[[t]]}$, and set $U=\mathbb P^1_{k}\smallsetminus\{P\}$.  Then by \cite[Theorem~5.9]{HarbaterHartmann:PatchingOverFields} and the discussion immediately after the proof of that theorem, we have fields 
	\begin{eqnarray*}
		F_U&=&\Frac(k[(x-a)^{-1}][[t]]) \\
		F_P&=&k((x-a,t)):=\operatorname{Frac}(k[[x-a,t]]) \\
		F_P^\circ&=&k((x-a))((t))
	\end{eqnarray*}
	such that the quadruple $(F,F_P,F_U,F_P^\circ)$ is a diamond with the factorization property. Note that  the derivation $\del=\del/\del x= \del/ \del(x-a)$ on $F$ extends to compatible derivations $\del/\del(x-a)$ on $F_U, F_P$ and $F_P^\circ$. Hence $(F,F_P,F_U,F_P^\circ)$ is in fact a differential diamond with the factorization property.
	
	To prove the theorem, it suffices to verify the hypotheses of Theorem~\ref{theo: main patching} for the fields $F_1=F_P$, $F_2=F_U$, $F_0=F_P^\circ$.    By \cite[Theorem~3.3]{BachmayrHarbaterHartmann:DifferentialEmbeddingProblemsOverLaurentSeriesFields}, there is
	a Picard-Vessiot ring $R_1/F_1$ with differential Galois group $N$ such that $R_1$ is a differential $F_1$-subalgebra of $F_0$.  Thus it remains to show that 
	the Picard-Vessiot ring $R$ over $F$ is a differential $F$-subalgebra of $F_1=F_P$.
	
	Recall that $R_0$ embeds into $k((x-a))$ as a differential $k(x)$-algebra; in particular, $R_0$ embeds into $k((x-a,t))$ as a differential $k(x)$-algebra, as does $F = k((t))(x) = k((t))(x-a)$.  So there is a canonical differential $F$-algebra homomorphism $R\cong R_0\otimes_{k(x)} k((t))(x)\to k((x-a,t))$; and since $R$ is a simple differential ring, this homomorphism is injective. Hence we may consider $R$ as a differential $F$-subalgebra of $k((x-a,t))=F_P$.	
\end{proof}

The following proposition shows that there is no harm in enlarging the constants when considering differential embedding problems over $k(x)$.

\begin{prop} \label{prop: descent for embedding problems}
	Let $k$ be an uncountable algebraically closed field of characteristic zero, let $(\alpha\colon G\twoheadrightarrow H,\ R)$ be an admissible differential embedding problem over $E:=k(x)$, and let $K/k$ be a field extension.
If the induced differential embedding problem $(\alpha_{K}\colon G_{K}\twoheadrightarrow H_{K},\ R\otimes_E F)$ over $F=K(x)$
is solvable, then so is $(\alpha\colon G\twoheadrightarrow H,\ R)$.
\end{prop}
	
\begin{proof}	
	After replacing $K$ by its algebraic closure (and using Lemma \ref{lemma: extend constants for PV ring}), we may assume that $K$ is algebraically closed.
	The proof proceeds in three steps. 
	We first show that the embedding problem $(\alpha\colon G\twoheadrightarrow H,\ R)$ descends to an embedding problem $(\alpha_0\colon G_0\twoheadrightarrow H_0,\ R_0)$ over $E_0=k_0(x)$, where $k_0\subseteq k$ is an algebraically closed field with $|k_0|<|k|$.
	Then, based on the assumption of the proposition, we show that for some algebraically closed field extension $k_1$ of $k_0$ inside $K$ with $|k_1|<|k|$, the embedding problem $(\alpha_1\colon G_1\twoheadrightarrow H_1,\ R_1)=((\alpha_0)_{k_1}\colon (G_0)_{k_1}\twoheadrightarrow (H_0)_{k_1},\ R_0\otimes_{E_0}E_1)$ over $E_1:=k_1(x)$ has a solution. Finally, we embed $k_1/k_0$ into $k$ and observe that the embedding problem  $((\alpha_1)_k\colon (G_1)_k\twoheadrightarrow (H_1)_k,\ R_1\otimes_{E_1} E)$ is the original embedding problem.

	Since $\rank(R)<|E|=|k|$, we know that $R$ is the Picard-Vessiot ring for a family of linear differential equations $(\partial(y)=A_{i}y,\ A_i\in E^{n_i\times n_i})_{i\in I}$ where $|I|=\rank(R)<|k|$. As $k$ is uncountable, there exists an algebraically closed subfield field $k_0$ of $k$ with $|k_0|<|k|$ such that all entries of all $A_i$'s lie in $E_0=k_0(x)\subseteq k(x)=E$. Let $R_0/E_0$ be the Picard-Vessiot ring for the family 
	\begin{equation} \label{eq: family for R}
	(\partial(y)=A_{i}y,\ A_i\in {E_0}^{n_i\times n_i})_{i\in I}
	\end{equation}
	 over $E_0$. Then $R$ and $R_0\otimes_{E_0}E$ both are Picard-Vessiot rings for the same family of equations, hence $R=R_0\otimes_{E_0}E$ and $H=(H_0)_{k}$ by Lemma \ref{lemma: extend constants for PV ring}, where $H_0:=\gal(R_0/E_0)$. 
	 Since $$\rank(G)=\rank(H)<|k|$$ by \cite[Lemma 3.6]{Wibmer:FreeProalgebraicGroups} (and admissibility), the coordinate ring of $k[G]$ can be generated by less than $|k|$ elements (cf. \cite[Prop. 3.1]{Wibmer:FreeProalgebraicGroups}). Thus, after enlarging $k_0$ if necessary, we can assume that $G=(G_0)_k$ for some proalgebraic group $G_0$ over $k_0$ and that $\alpha=({\alpha_0})_k$ for some $\alpha_0\colon G_0\twoheadrightarrow H_0$.

	By assumption, there exists a Picard-Vessiot ring $S'/F$ containing $R'$ and an isomorphism $\psi'\colon G_K\to\gal(S'/F)$ such that  	
	\begin{equation} \label{eq:diagram commutes over K}
	\xymatrix{
		\gal(S'/F) \ar@{->>}[r] & \gal(R'/F) \\
		G_K \ar^-{\psi'}[u] \ar@{->>}^{\alpha_K}[r] & H_K	\ar[u]
	}
	\end{equation}
	commutes.

	As $\rank(S')=\rank(G_K)=\rank(G)<|k|$, there exist a family of linear differential equations 
	\begin{equation} \label{eq: family for S'}
(\partial(y)=B_{j}y,\ B_j\in F^{n_j\times n_j})_{j\in J}
	\end{equation} where $|J|<|k|$, such that $S'/F$ is the Picard-Vessiot ring for this family. As $R_0\otimes_{E_0}F=R'\subseteq S'$ we may assume without loss of generality that the family (\ref{eq: family for R}) is a subfamily of 
	$(\partial(y)=B_{j}y,\ B_j\in F^{n_j\times n_j})_{j\in J}$.
	
	There exists an algebraically closed field extension $k_1/k_0$ inside $K$ of cardinality less than $|k|$ such that all entries of all $B_j$'s lie in $E_1=k_1(x)\subseteq K(x)=F$. Let $S_1/E_1$ be the Picard-Vessiot ring for the family $(\partial(y)=B_{j}y,\ B_j\in E_1^{n_j\times n_j})_{j\in J}$ over $E_1$. Then $S'=S_1\otimes_{E_1}F$ and with $G_1:=(G_0)_{k_1}$ we obtain that
	$$\psi'\colon (G_1)_{K}=G_K\cong \gal(S'/F)=\gal(S_1/E_1)_{K}$$
	 is an isomorphism of proalgebraic groups over $K$. As $\rank(G_1)=\rank(G)<|k|$, after enlarging $k_1$ if necessary, we can assume that, $\psi'=(\psi_1)_{K}$
	for some isomorphism $\psi_1\colon G_1\to \gal(S_1/E_1)$ of proalgebraic groups over $k_1$. Set $R_1=R_0\otimes_{E_0} E_1$,  $H_1=(H_0)_{k_1}$ and $\alpha_1=(\alpha_0)_{k_1}$. We have $R_0\subseteq R'\subseteq S'=S_1\otimes_{E_1} F$. As the family (\ref{eq: family for S'}) contains the family (\ref{eq: family for R}), we see that $R_1=R_0\otimes_{E_0}E_1\subseteq S_1$. 
	The diagram	 
	\begin{equation} \label{eq:diagram commutes over k1}
	\xymatrix{
		\gal(S_1/E_1) \ar@{->>}[r] & \gal(R_1/E_1) \\
		G_1 \ar^{\psi_1}[u] \ar@{->>}^{\alpha_1}[r] & H_1	\ar[u]
	}
	\end{equation}
	of proalgebraic groups over $k_1$ becomes diagram (\ref{eq:diagram commutes over K}) after base extension form $k_1$ to $K$. As diagram (\ref{eq:diagram commutes over K}) commutes, diagram (\ref{eq:diagram commutes over k1}) also commutes.

	Since $k$ is uncountable and $|k_0|<|k|$, we see that $\trdeg(k/k_0)=|k|>|k_1|$. Thus there exists an embedding of $k_1$ into $k$ over $k_0$. Now, if we consider $k_1$ as a subfield of $k$ and $E_1$ as a subfield of $E$ via this embedding, then $S=S_1\otimes_{E_1}E$ is a Picard-Vessiot ring over $E$ that contains $R=R_1\otimes_{E_1}E$. Moreover, the diagram 
	$$
	\xymatrix{
		\gal(S/E) \ar@{->>}[r] & \gal(R/E) \\
		G \ar^{(\psi_1)_k}[u] \ar@{->>}^{\alpha}[r] & H	\ar[u]
	}
	$$
	of proalgebraic groups over $k$ commutes, because it is the base change to $k$ of the commutative diagram (\ref{eq:diagram commutes over k1}).
	\end{proof}

We are finally prepared to prove Matzat's conjecture. Our proof works for any algebraically closed field of characteristic zero of infinite transcendence degree over $\mathbb{Q}$ in place of $\C$.

\begin{theo}[Matzat's conjecture for infinite transcendence degree]\label{result}
	Let $k$ be an algebraically closed field of characteristic zero of infinite transcendence degree over $\mathbb{Q}$. Then the absolute differential Galois group of $k(x)$ is the free proalgebraic group on a set of cardinality $|k|$.
\end{theo}

\begin{proof} 
If $k$ is countable, this was shown in \cite[Theorem~3.10]{BachmayrHarbaterHartmannWibmer:FreeDifferentialGaloisGroups}. So we may henceforth assume that $k$ is uncountable. Define $E=k(x)$. By Proposition \ref{prop: criterion for abs differential group to be free}, it suffices to show that differential embedding problems of type~\ref{free i} and~\ref{free ii} as in that proposition are solvable. As differential embeddings of type~\ref{free i} (i.e., those of finite type) were solved in \cite[Cor. 4.6]{BachmayrHartmannHarbaterPopLarge}, it suffices to show that those of type~\ref{free ii} are solvable. So let $(\alpha\colon G\twoheadrightarrow H,\ R)$ be a split admissible differential embedding problem over $E=k(x)$. Set $K=k((t))$ and $F=K(x)$. By Theorem \ref{theo: solve F/E DEB} the induced differential embedding problem $(\alpha_{K}\colon G_{K}\twoheadrightarrow H_{K},\ R\otimes_E F)$ is solvable. So it follows from Proposition \ref{prop: descent for embedding problems} that $(\alpha\colon G\twoheadrightarrow H,\ R)$ has a solution.
\end{proof}

In the proof of Theorem \ref{result}, it was sufficient to consider only certain types of admissible differential embedding problems, but by \cite[Theorem 3.7]{BachmayrHarbaterHartmannWibmer:FreeDifferentialGaloisGroups} the result of Theorem \ref{result} then actually  implies:

\begin{cor} \label{cor: cor to main}
Let $k$ be an algebraically closed field of infinite transcendence degree over $\mathbb{Q}$. Then every differential embedding problem $(G\twoheadrightarrow H, R)$ over $k(x)$ with $\rank(G)\leq |k|$ and $\rank(R)<|k|$ has a solution. 
\end{cor}
In particular, every admissible differential embedding problem over $k(x)$ has a solution. More is true:

\begin{cor}
Let $k$ be an algebraically closed field of infinite transcendence degree over $\mathbb{Q}$. Then every non-trivial differential embedding problem $(G\twoheadrightarrow H, R)$ of finite type over $k(x)$ has $|k|$ independent solutions, i.e., there exist $|k|$ solutions $S_i$ (inside the Picard-Vessiot extension of the family of all linear differential equations over $k(x)$) such that the fraction fields of the $S_i$'s are linearly independent over the field of fractions of $R$. 
\end{cor}

\begin{rem} \label{rem: Matzat implies Douady}
Matzat's conjecture implies Douady's theorem. More generally, if $F$ is a differential field with $k=F^\del$ algebraically closed such that the absolute differential Galois group of $F$ is the free proalgebraic group on a set of cardinality $|F|$, then the absolute Galois group of $F$ is the free profinite group on a set of cardinality $|F|$.
\end{rem}
\begin{proof} For a proalgebraic group $G$, let $\pi_0(G)$ denote the maximal pro-\'{e}tale quotient of $G$, i.e., $\pi_0(G)=G/G^o$. Let $\Gamma$ be the free proalgebraic group (over $k$) on the set $X$ with structure map $\iota\colon X\to \Gamma(k)$. Then $\iota_0\colon X\xrightarrow{\iota}\Gamma(k)\to \pi_0(\Gamma)(k)$ converges to~$1$ (see \cite[Lemma 2.12]{Wibmer:FreeProalgebraicGroups}); and from the universal properties of $\iota\colon X\to\Gamma(k)$ and $\Gamma\to\pi_0(\Gamma)$ it follows that $\iota_0$ satisfies the following universal property: If $G$ is a pro-\'{e}tale proalgebraic group and $\varphi\colon X\to G(k)$ a map converging to $1$, then there exists a unique morphism $\psi\colon \pi_0(\Gamma)\to G$ such that
$$
\xymatrix{
X \ar^{\iota_0}[rr] \ar_\varphi[rd] & & \pi_0(\Gamma)(k) \ar^-{\psi_k}[ld] \\
& G(k) &
}
$$
commutes. In the language of \cite{Wibmer:FreeProalgebraicGroups}, this means that $\pi_0(\Gamma)$ is the free pro-$\mathcal{C}$-group (over $k$) on the set $X$, where $\mathcal{C}$ is the class of all \'{e}tale algebraic groups over $k$.

As $k$ is algebraically closed, the category of (abstract) finite groups is equivalent to the category of \'{e}tale algebraic groups (over $k$). This extends to an equivalence between the category of profinite groups and the category of pro-\'{e}tale algebraic groups (over $k$). Under this equivalence $\pi_0(\Gamma)$ corresponds to the free profinite group on the set $X$. The claim now follows from the fact that $\pi_0$ of the absolute differential Galois group of $F$ is the differential Galois group of $\bar F/F$, where $\bar F$ is the algebraic closure of $F$.
\end{proof}

To resolve the remaining case of Matzat's conjecture (when $k$ has finite transcendence degree over ${\mathbb Q}$ it would suffice to show that every differential embedding problem of finite type over $k(x)$ has a solution (\cite[Cor.~3.9]{BachmayrHarbaterHartmannWibmer:FreeDifferentialGaloisGroups}).

\medskip

\footnotesize

\noindent Author information:

\medskip

\noindent Annette Bachmayr (n\'{e}e Maier): 
Institute of Mathematics, University of Mainz, Staudingerweg 9, 55128 Mainz, Germany.
E-mail: {\tt abachmay@uni-mainz.de}

\medskip

\noindent David Harbater: Department of Mathematics, University of Pennsylvania, Philadelphia, PA 19104-6395, USA. E-mail: {\tt harbater@math.upenn.edu}

\medskip

\noindent Julia Hartmann:  Department of Mathematics, University of Pennsylvania, Philadelphia, PA 19104-6395, USA. E-mail: {\tt hartmann@math.upenn.edu}

\medskip

\noindent Michael Wibmer: Institute of Analysis and Number Theory, Graz University of Technology, Kopernikusgasse 24, 8010 Graz, Austria.
E-mail: {\tt wibmer@math.tugraz.at}  

\end{document}